\documentclass[12pt]{amsart}
\usepackage{amssymb}
\usepackage[all,cmtip]{xy}
\usepackage{hyperref}
\usepackage{fullpage}
\usepackage{enumerate}
\usepackage{tikz}
\usepackage{thmtools}

\numberwithin{equation}{section}

\theoremstyle{plain}
\newtheorem{lemma}[equation]{Lemma}
\newtheorem{theorem}[equation]{Theorem}
\newtheorem{proposition}[equation]{Proposition}

\theoremstyle{definition}

\theoremstyle{remark} 
\newtheorem{remark}[equation]{Remark}

\newcommand{\End}{\mathsf{End}}
\newcommand{\ep}{\varepsilon}
\newcommand{\Ext}{\mathsf{Ext}}
\renewcommand{\le}{\leqslant}
\renewcommand{\ge}{\geqslant}
\newcommand{\Rad}{\mathsf{Rad}}
\newcommand{\Soc}{\mathsf{Soc}}
\newcommand{\bZ}{\mathbb Z}
\newcommand{\PSL}{\mathrm{PSL}}
\newcommand{\PGL}{\mathrm{PGL}}
\newcommand{\Aut}{\mathrm{Aut}}
\newcommand{\Sz}{\mathrm{Sz}}
\newcommand{\GL}{\mathrm{GL}}
\newcommand{\TT}{\mathrm{t}}
\newcommand{\Irr}{\mathrm{Irr}}
\newcommand{\IBr}{\mathrm{IBr}}
\renewcommand{\phi}{\varphi}

\author{Dave Benson} 
\address{Institute of Mathematics \\ 
University of Aberdeen \\ 
Aberdeen AB24 3UE \\ 
United Kingdom}
\author{Benjamin Sambale}
\address{Institut f\"ur Algebra, Zahlentheorie und Diskrete Mathematik\\
Leibniz Universit\"at Hannover\\
Welfengarten 1\\ 
30167 Hannover, Germany}

\title{Finite dimensional algebras not arising as blocks\\ of group algebras}



\begin{document}
\frenchspacing

\begin{abstract}
We develop new techniques to classify basic algebras of blocks of
finite groups over algebraically closed fields of prime
characteristic. 
We apply these techniques to simplify and extend previous
classifications by Linckelmann, Murphy and Sambale. In particular, we fully classify blocks with
$16$-dimensional basic algebra. 
\end{abstract}

\maketitle
\renewcommand{\sectionautorefname}{Section}
\section{Introduction}

Linckelmann~\cite{Linckelmann:2018b} instigated the study of small dimensional
symmetric basic algebras over an algebraically closed field of prime characteristic,
in the context of enumerating which of them could be 
the basic algebra of a block of a finite group.
In that paper, he gave a complete classification up to dimension twelve, 
except for one case of an algebra of dimension nine; see Section~2.9 of that 
paper. The paper of Linckelmann and Murphy~\cite{Linckelmann/Murphy:2021a}
eliminated that $9$-dimensional 
algebra using some fairly sophisticated group representation
theory. We provide a proof of this elimination that is completely different
from the one in that paper, by examining the Auslander--Reiten quiver.  

The second author~\cite{Sambale:2021a} took Linckelmann's methods further, up
to dimension fourteen. We provide alternative proofs for some of the difficult
cases that occurred in that paper. By a recent paper of Macgregor~\cite{Macgregor:2022a}, it became clear that the classification of tame basic algebras in dimension $14$ might be incomplete. Hence, the list given in \cite{Sambale:2021a} might be incomplete as well. We comment on the details in \autoref{sec:small} below. 

By way of preparation, we prove some theorems that dispose of a number of
possible Cartan matrices. The most interesting of these is the
following theorem, whose proof can be found in \autoref{se:Cartan}.

\begin{theorem}\label{th:310}
Suppose that $A$ is a finite dimensional indecomposable symmetric
algebra over an algebraically closed field. Suppose that $A$ 
has a simple module $S$ whose Cartan invariant $c_{S,S}$ is
equal to $3$ and all the other $c_{S,T}$ are either one or zero. 
Then $A$ is not Morita equivalent to a block
of wild representation type of a finite group algebra in prime characteristic.
\end{theorem}

It has been observed in \cite{Sambale:2021a} that in dimension $15$
there might be a $13$-block of defect $1$, which is not known to
exist. Leaving this open case aside, we show that there is only one more block in this dimension.

\begin{theorem}\label{th:dim15}
Let $B$ be a block of a finite group with defect group $D$ and basic
algebra $A$ of dimension $15$. Then one of the following holds:
\begin{enumerate}[(1)]
\item $D\cong C_{19}$ and $A$ is Morita equivalent to the principal $19$-block of $\GL(3,7)$. 
\item $D\cong C_{13}$ and $A$ is a Brauer tree algebra with Cartan matrix   
\[\begin{pmatrix}
5&1&1\\1&2&1\\1&1&2
\end{pmatrix}.\] 
\end{enumerate}
\end{theorem}

Finally, we extend the classification of basic algebras to dimension
$16$ as follows. 

\begin{theorem}\label{th:dim16}
Let $B$ be a block of a finite group with defect group $D$ and basic
algebra $A$ of dimension $16$.  
Then one of the following holds: 
\begin{enumerate}[\rm (1)]
\item $|D|=16$ and $A$ is isomorphic to the group algebra of $D$.\smallskip
\item $D\cong C_2^4$ and $A$ is Morita equivalent to a non-principal
  block of $H\cong D\rtimes 3^{1+2}_+$ with $H/Z(H)\cong A_4^2$.\smallskip
\item $D\cong C_{23}$ and $A$ is Morita equivalent to the principal
  $23$-block of $\PSL(2,137)$.\smallskip
\item $D\cong C_5$ and $A$ is Morita equivalent to the principal
  $5$-block of $S_5$ or $\Sz(8)$.\smallskip
\item $D\cong C_{13}$ and $A$ is Morita equivalent to the principal
  $13$-block of $\GL(4,5)$.\smallskip
\item $D\cong D_8$ and $A$ is Morita equivalent to the principal
  $2$-block of $\GL(3,2)$.  
\end{enumerate}
In total there are 20 Morita equivalence classes. 
\end{theorem}

\section{Preliminaries}

Throughout, we work with a finite group $G$ over an algebraically 
closed field $k$ of characteristic $p$ dividing $|G|$. We fix a block
$B$ of $kG$ with defect group $D$. Then the basic algebra $A$ of $B$
is a finite dimensional symmetric algebra. Recall that $A$ and $B$
have isomorphic centres.  
The dimension of this centre coincides with the number $k(B)$ of
irreducible characters in $B$. The number of simple modules of $B$
(and $A$) is denoted by $l(B)$.  
The dimension of $A$ itself is the sum of the entries of the Cartan matrix $C$
of $B$. The determinant of $C$ is a power of $p$, which
severely restricts the possibilities. The largest elementary divisor
of $C$ is $|D|$ and it occurs with multiplicity one. If $\det(C)=p$,
we conclude that $D$ is cyclic of order $p$. In this case, $A$ is a
Brauer tree algebra and of finite representation type. 
This further limits the possibilities for $C$. We extend Proposition~2
in \cite{Sambale:2021a} as follows (the proof is the same). 

\begin{proposition}\label{propfinite}
Let $B$ be a block with cyclic defect group $D$, $l(B)=4$ and
multiplicity $m:=\frac{|D|-1}{4}$. Then the possible Brauer trees are
given as follows: 
\begin{enumerate}[\rm (1)]
\item \begin{align*}
\begin{tikzpicture}[thick,baseline=(a.center)]
\node[draw,circle,fill,scale=.5] (a) {};
\node[draw,right of=a,circle,scale=.5] (b) {};
\node[draw,above of=a,circle,scale=.5] (c) {};
\node[draw,below of=a,circle,scale=.5] (d) {};
\node[draw,left of=a,circle,scale=.5] (e) {};
\draw (b)--(a)--(c) (d)--(a)--(e);
\end{tikzpicture}&&
C=\begin{pmatrix}
m+1&m&m&m\\m&m+1&m&m\\m&m&m+1&m\\m&m&m&m+1
\end{pmatrix}&&
\dim A=16m+4=4|D|
\end{align*}
This occurs for $B=k[D\rtimes C_4]$ provided $4\mid p-1$. 

\item \begin{align*}
\begin{tikzpicture}[thick,baseline=(a.center)]
\node[draw,circle,scale=.5] (a) {};
\node[draw,right of=a,circle,scale=.5] (b) {};
\node[draw,above of=a,circle,scale=.5] (c) {};
\node[draw,below of=a,circle,scale=.5] (d) {};
\node[draw,left of=a,fill,circle,scale=.5] (e) {};
\draw (b)--(a)--(c) (d)--(a)--(e);
\end{tikzpicture}&&
C=\begin{pmatrix}
m+1&1&1&1\\1&2&1&1\\1&1&2&1\\1&1&1&2
\end{pmatrix}&&
\dim A=m+19=\frac{|D|+75}{4}
\end{align*}

\item \begin{align*}
\begin{tikzpicture}[thick,baseline=(a.center)]
\node[draw,fill,circle,scale=.5] (a) {};
\node[draw,above of=a,circle,scale=.5] (c) {};
\node[draw,below of=a,circle,scale=.5] (d) {};
\node[draw,right of=a,circle,scale=.5] (e) {};
\node[draw,right of=e,circle,scale=.5] (f) {};
\draw (c)--(a)--(d) (a)--(e)--(f);
\end{tikzpicture}&&
C=\begin{pmatrix}
m+1&m&m&.\\m&m+1&m&.\\m&m&m+1&1\\.&.&1&2
\end{pmatrix}&&
\dim A=9m+7=\frac{9|D|+19}{4}
\end{align*}

\item \begin{align*}
\begin{tikzpicture}[thick,baseline=(a.center)]
\node[draw,circle,scale=.5] (a) {};
\node[draw,above of=a,circle,scale=.5] (c) {};
\node[draw,below of=a,circle,scale=.5] (d) {};
\node[draw,fill,right of=a,circle,scale=.5] (e) {};
\node[draw,right of=e,circle,scale=.5] (f) {};
\draw (c)--(a)--(d) (a)--(e)--(f);
\end{tikzpicture}&&
C=\begin{pmatrix}
m+1&m&1&.\\m&m+1&1&.\\1&1&2&1\\.&.&1&2
\end{pmatrix}&&
\dim A=4m+16=|D|+15
\end{align*}

\item \begin{align*}
\begin{tikzpicture}[thick,baseline=(a.center)]
\node[draw,circle,scale=.5] (a) {};
\node[draw,fill,above of=a,circle,scale=.5] (c) {};
\node[draw,below of=a,circle,scale=.5] (d) {};
\node[draw,right of=a,circle,scale=.5] (e) {};
\node[draw,right of=e,circle,scale=.5] (f) {};
\draw (c)--(a)--(d) (a)--(e)--(f);
\end{tikzpicture}&&
C=\begin{pmatrix}
m+1&1&1&.\\1&2&1&.\\1&1&2&1\\.&.&1&2
\end{pmatrix}&&
\dim A=m+15=\frac{|D|+59}{4}
\end{align*}

\item 
\begin{align*}
\begin{tikzpicture}[thick]
\node[draw,circle,fill,scale=.5] (a) {};
\node[draw,right of=a,circle,scale=.5] (b) {};
\node[draw,right of=b,circle,scale=.5] (c) {};
\node[draw,right of=c,circle,scale=.5] (d) {};
\node[draw,right of=d,circle,scale=.5] (e) {};
\draw (a)--(b)--(c)--(d)--(e);
\end{tikzpicture}&&
C=\begin{pmatrix}
m+1&1&.&.\\1&2&1&.\\.&1&2&1\\.&.&1&2
\end{pmatrix}&&
\dim A=m+13=\frac{|D|+51}{4}.
\end{align*}

\item 
\begin{align*}
\begin{tikzpicture}[thick]
\node[draw,circle,scale=.5] (a) {};
\node[draw,right of=a,fill,circle,scale=.5] (b) {};
\node[draw,right of=b,circle,scale=.5] (c) {};
\node[draw,right of=c,circle,scale=.5] (d) {};
\node[draw,right of=d,circle,scale=.5] (e) {};
\draw (a)--(b)--(c)--(d)--(e);
\end{tikzpicture}&&
C=\begin{pmatrix}
m+1&m&.&.\\m&m+1&1&.\\.&1&2&1\\.&.&1&2
\end{pmatrix}&&
\dim A=4m+10=|D|+9.
\end{align*}

\item 
\begin{align*}
\begin{tikzpicture}[thick]
\node[draw,circle,scale=.5] (a) {};
\node[draw,right of=a,circle,scale=.5] (b) {};
\node[draw,right of=b,fill,circle,scale=.5] (c) {};
\node[draw,right of=c,circle,scale=.5] (d) {};
\node[draw,right of=d,circle,scale=.5] (e) {};
\draw (a)--(b)--(c)--(d)--(e);
\end{tikzpicture}&&
C=\begin{pmatrix}
2&1&.&.\\1&m+1&m&.\\.&m&m+1&1\\.&.&1&2
\end{pmatrix}&&
\dim A=4m+10=|D|+9.
\end{align*}
\end{enumerate}
\end{proposition}

Sometimes a Cartan matrix leads to a Brauer graph 
algebra, which is the same as a symmetric special biserial algebra.
They all have finite or tame representation type by 
Wald--Waschb\"usch~\cite{Wald/Waschbusch:1985a}.

Blocks of finite group algebras with tame (but not finite) representation type
only occur in characteristic two, and the defect groups in this case
are dihedral, semidihedral or generalised quaternion. 
These algebras were first investigated by Erdmann~\cite{Erdmann:1990a}. 
By a recent paper of Macgregor~\cite{Macgregor:2022a}, all Cartan matrices 
of tame blocks are known and we list them for the convenience of the 
reader (this includes the degenerate case $D\cong C_2^2$, although 
it is not listed in \cite{Macgregor:2022a}):

\begin{theorem}\label{proptame}
Let $B$ be a non-nilpotent tame block with defect group $D$ 
of order $2^n$ and Cartan matrix $C$. Then one of the following holds:
\begin{enumerate}[\rm (1)]
\item $D\cong D_{2^n}$ and $C$ is one of the following:
\begin{gather*}
\begin{pmatrix}
2^n&2^{n-1}\\2^{n-1}&2^{n-2}+1
\end{pmatrix},\quad
\begin{pmatrix}
4&2\\2&2^{n-2}+1
\end{pmatrix},\quad
\begin{pmatrix}
2^n&2^{n-1}&2^{n-1}\\
2^{n-1}&2^{n-2}+1&2^{n-2}\\
2^{n-1}&2^{n-2}&2^{n-2}+1
\end{pmatrix}, \\
\begin{pmatrix}
2&1&1\\
1&2^{n-2}+1&2^{n-2}\\
1&2^{n-2}&2^{n-2}+1
\end{pmatrix},\quad
\begin{pmatrix}
4&2&2\\
2&2^{n-2}+1&1\\
2&1&2
\end{pmatrix}
\end{gather*}

\item $D\cong Q_{2^n}$ and $C$ is one of the following:
\begin{gather*}
2\begin{pmatrix}
2^{n-1}&2^{n-2}\\
2^{n-2}&2^{n-3}+1
\end{pmatrix},\quad
2\begin{pmatrix}
4&2\\2&2^{n-3}+1
\end{pmatrix},\quad
2
\begin{pmatrix}
2^{n-1}&2^{n-2}&2^{n-2}\\
2^{n-2}&2^{n-3}+1&2^{n-3}\\
2^{n-2}&2^{n-3}&2^{n-3}+1
\end{pmatrix},\\
2\begin{pmatrix}
2&1&1\\
1&2^{n-3}+1&2^{n-3}\\
1&2^{n-3}&2^{n-3}+1
\end{pmatrix},\quad
2\begin{pmatrix}
4&2&2\\
2&2^{n-3}+1&1\\
2&1&2
\end{pmatrix}
\end{gather*}

\item $D\cong SD_{2^n}$ and $C$ is one of the following:
\begin{gather*}
2\begin{pmatrix}
2^{n-1}&2^{n-2}\\
2^{n-2}&2^{n-3}+1
\end{pmatrix},\quad
2\begin{pmatrix}
4&2\\
2&2^{n-3}+1
\end{pmatrix},\quad
\begin{pmatrix}
2^n&2^{n-1}\\
2^{n-1}&2^{n-2}+1
\end{pmatrix},\quad
\begin{pmatrix}
4&2\\
2&5
\end{pmatrix},\\
\begin{pmatrix}
4&2&2\\
2&2^{n-2}+1&1\\
2&1&3
\end{pmatrix},\quad
\begin{pmatrix}
2^n&2^{n-1}&2^{n-1}\\
2^{n-1}&2^{n-2}+1&2^{n-2}\\
2^{n-1}&2^{n-2}&2^{n-2}+2
\end{pmatrix},\quad
\begin{pmatrix}
2^{n-2}+1&2^{n-2}-1&2^{n-2}\\
2^{n-2}-1&2^{n-2}+1&2^{n-2}\\
2^{n-2}&2^{n-2}&2^{n-2}+2
\end{pmatrix},\\
\begin{pmatrix}
8&4&4\\
4&2^{n-2}+2&2\\
4&2&3
\end{pmatrix},\quad
\begin{pmatrix}
3&1&2\\
1&3&2\\
2&2&2^{n-2}+2
\end{pmatrix}.
\end{gather*}
\end{enumerate}
\end{theorem}

We apply the previous theorem to list all Cartan matrices of $2$-blocks with defect at most three.

\begin{proposition}\label{prop8}
Let $B$ be a non-nilpotent $2$-block with defect group $D$ and Cartan matrix $C$. If $|D|\le 8$, then one of the following holds:
\begin{enumerate}[\rm (1)]
\item $B$ is tame and $C$ is one of the following matrices:
\begin{gather*}
\begin{pmatrix}
8&4\\
4&3
\end{pmatrix},\quad
\begin{pmatrix}
4&2\\
2&3
\end{pmatrix},\quad
\begin{pmatrix}
2&1&1\\
1&2&1\\
1&1&2
\end{pmatrix},\quad
\begin{pmatrix}
4&2&2\\
2&2&1\\
2&1&2
\end{pmatrix},\quad
\begin{pmatrix}
8&4&4\\
4&3&2\\
4&2&3
\end{pmatrix},\\
\begin{pmatrix}
2&1&1\\
1&3&2\\
1&2&3
\end{pmatrix},\quad
\begin{pmatrix}
4&2&2\\
2&3&1\\
2&1&3
\end{pmatrix},\quad
2\begin{pmatrix}
4&2&2\\
2&2&1\\
2&1&2
\end{pmatrix},\quad
2\begin{pmatrix}
2&1&1\\
1&2&1\\
1&1&2
\end{pmatrix}.
\end{gather*}

\item $D\cong C_2^3$ and $C$ is one of the following matrices:
\begin{gather*}
2\begin{pmatrix}
2&1&1\\
1&2&1\\
1&1&2
\end{pmatrix},\quad
2\begin{pmatrix}
4&2&2\\
2&2&1\\
2&1&2
\end{pmatrix},\quad
\begin{pmatrix}
8&6&2&2&2\\
6&8&2&2&2\\
2&2&4&.&.\\
2&2&.&4&.\\
2&2&.&.&4
\end{pmatrix},\quad
\begin{pmatrix}
8&4&4&4&4\\
4&4&3&3&1\\
4&3&4&2&2\\
4&3&2&4&2\\
4&1&2&2&4
\end{pmatrix},\\
\begin{pmatrix}
8&4&4&4&3\\
4&4&2&2&2\\
4&2&4&2&2\\
4&2&2&4&2\\
3&2&2&2&2
\end{pmatrix},\quad
\begin{pmatrix}
4&2&2&2&2&2&2\\
2&4&2&2&2&2&2\\
2&2&4&2&2&2&2\\
2&2&2&4&2&2&2\\
2&2&2&2&4&2&2\\
2&2&2&2&2&4&2\\
2&2&2&2&2&2&4
\end{pmatrix},\quad
\begin{pmatrix}
8&4&4&4&2&2&2\\
4&4&2&2&.&2&1\\
4&2&4&2&1&.&2\\
4&2&2&4&2&1&.\\
2&.&1&2&2&.&.\\
2&2&.&1&.&2&.\\
2&1&2&.&.&.&2
\end{pmatrix}.
\end{gather*}
\end{enumerate}
\end{proposition}
\begin{proof}
If $D\cong\{1,C_2,C_4,C_8,C_4\times C_2\}$, then $B$ is nilpotent
since $\Aut(D)$ is a $2$-group. If $D\cong\{C_2^2,D_8,Q_8\}$, then $B$
is tame and the claim follows from \autoref{proptame} (note that
$l(B)=2$ is only possible for $D\cong D_8$). Finally, if $D\cong
C_2^3$, then the claim follows from Eaton~\cite{Eaton:2016a}.
\end{proof}

\section{Auslander--Reiten theory}

In this section, we collect some well known results from Auslander--Reiten 
theory for blocks of finite groups, that we shall use in this paper.

\begin{theorem}\label{th:Webb}
The tree class of every component of the Auslander--Reiten quiver
of $B$ is either
the Dynkin diagram $A_n$, in which case $B$ has cyclic defect,
or a Euclidean diagram, or one of three infinite trees, $A_\infty$,
$D_\infty$ or $A^\infty_\infty$.
\end{theorem}
\begin{proof}
This is Theorem~A of Webb~\cite{Webb:1982b}. Because the field
$k$ is algebraically closed, the infinite trees $B_\infty$ and $C_\infty$ 
do not occur.
\end{proof}

\begin{theorem}\label{th:Bessenrodt}
If the tree class of a component of the Auslander--Reiten quiver of $B$ is a
Euclidean diagram, then $B$ has Klein four defect group, and the tree class
is $\tilde A_{1,2}$ or $\tilde A_5$.
\end{theorem}
\begin{proof}
This is Theorem~1.1 of Bessenrodt~\cite{Bessenrodt:1991a}. Since $k$
is algebraically closed, $\tilde B_3$ does not occur.
\end{proof}

\begin{theorem}\label{th:Erdmann}
Every Auslander--Reiten component of a block $B$ of wild representation
type has tree class $A_\infty$.
If $B$ has an Auslander--Reiten component of tree class $D_\infty$, 
then $B$ is a tame block with semidihedral defect groups.
\end{theorem}
\begin{proof}
This is Theorem~1 of Erdmann~\cite{Erdmann:1995a}.
\end{proof}

\begin{theorem}\label{th:summands}
If $P$ is a projective indecomposable module in the block $B$,
then the radical modulo the socle, $\Rad(P)/\Soc(P)$ has at most two
direct summands. 
\end{theorem}
\begin{proof}
This follows from the fact that there is an almost split sequence
of the form
\[ 0 \to \Rad(P) \to P \oplus \Rad(P)/\Soc(P) \to P/\Soc(P) \to 0. \]
If $\Rad(P)/\Soc(P)$ has more than two direct summands, 
then 
by Theorems~\ref{th:Webb} and~\ref{th:Bessenrodt}, the only way for
this to be part of an Auslander--Reiten component is for the
tree class to be $D_\infty$. By \autoref{th:Erdmann},
this implies that $B$ has semidihedral defect groups. 
Examining Erdmann~\cite{Erdmann:1988a,Erdmann:1990c},
there are no examples with semidihedral defect groups
where $\Rad(P)/\Soc(P)$ has more than two summands.
\end{proof}

\begin{remark}
Without assuming that the field $k$ is algebraically closed, there
are examples where $\Rad(P)/\Soc(P)$ has three summands.
But this only happens when the defect group is a Klein four group, 
$k$ does not have a primitive cube root of unity, and $B$ is
Morita equivalent to the principal block of $A_4$ or $A_5$
(see Bessenrodt~\cite{Bessenrodt:1991a}).
\end{remark}

\section{Small symmetric local algebras}

We shall need the following 
facts about small symmetric local algebras.

\begin{proposition}\label{pr:le7}
If $A$ be a symmetric local $k$-algebra with $\dim_k A \le 7$,
then 
\begin{enumerate}[\rm (1)]
\item
$A$ is commutative.
\item 
If $\Rad(A)/\Soc(A)$ is indecomposable, then 
one of the following is true.

{\rm (a)} 
$A\cong k[x]/(x^n)$ for some $n\le 7$, 

{\rm (b)} 
$A\cong k[x,y]/(x^3,y^2)$, of dimension $6$, 

{\rm (c)} $k$ has characteristic two and  
$A\cong k[x,y]/(x^3,y^2+x^2y)$ of dimension $6$, or

{\rm (d)} $k$ has characteristic three and $A\cong
k[x,y]/(x^3+x^2y,y^2)$ of dimension $6$. 
\end{enumerate}
\end{proposition}
\begin{proof}\hfill
\begin{enumerate}[(1)]
\item Let $Z$ be the centre of $A$. 
K\"ulshammer~\cite{Kulshammer:1980a} proved that if $\dim_k Z \le 4$, then $A$ 
is commutative. This was extended by 
Chlebowitz and K\"ulshammer~\cite{Chlebowitz/Kulshammer:1992a}, 
where it is proved that if $\dim_k Z =5$, then $A$ has dimension $5$ or $8$.
It cannot happen that $\dim_k Z=6$ and $\dim_k A = 7$, so this 
proves that $A$ is commutative.

\item 
Let $X$ denote $\Rad(A)/\Soc(A)$.
Poonen~\cite{Poonen:2008a} lists the commutative 
local algebras of dimension up to six. Among these, the ones 
that are Gorenstein with 
$X$ indecomposable are those listed above (note that the
algebras listed in cases (c) and (d) are isomorphic to $k[x,y]/(x^3,y^2)$ 
in other characteristics). 
It remains to deal with dimension seven. 
If the radical layers of $A$ have dimensions 
$[1,1,1,1,1,1,1]$ we are in case (a). If they have dimensions
$[1,2,1,1,1,1]$, then $\Soc(X)$ has dimension two while $\Rad^3(X)$
has dimension one. An element of $\Soc(X)$ that is not in $\Rad^3(X)$
spans a $1$-dimensional summand of $X$, so $X$ decomposes. Similarly,
in the case $[1,3,1,1,1]$, an element of $\Soc(X)$ that is not in
$\Rad^2(X)$ spans a $1$-dimensional summand. In the cases
$[1,4,1,1]$ and $[1,3,2,1]$, an element of $\Soc(X)$ that is not in
$\Rad(X)$ spans a $1$-dimensional summand.
In the case $[1,5,1]$, $X$ is semisimple, and decomposes as a direct
sum of five $1$-dimensional summands.
In the remaining case $[1,2,2,1,1]$, $A/\Soc(A)$ is a $6$-dimensional
algebra with radical layers $[1,2,2,1]$ and socle layers $[1,1,2,2]$.
Again examining Poonen's list~\cite{Poonen:2008a}, the possibilities
for $A/\Soc(A)$ are $k[x,y]/(x^2,xy^2,y^4)$ and
$k[x,y]/(x^2+y^3,xy^2,y^4)$. In both these cases, the quotient 
$\Rad(A)/(\Soc(A),xA)$ of $X$
is uniserial of length three spanned by the powers of $y$, but $X$
has no uniserial submodule of length three. This contradicts the fact
that $X$ is supposed to be self-dual, since $A$ is Gorenstein.\qedhere
\end{enumerate}
\end{proof}

\section{Cartan invariants}\label{se:Cartan}

In this section, we prove some theorems about Cartan invariants of
blocks of group algebras. 

\begin{proof}[Proof of \autoref{th:310}]
Suppose that $A$ is as
in the theorem, and that $A$ has wild representation type.
We examine the structure of the projective
cover $P_S$ of $S$. It follows from \autoref{th:summands}
that $\Rad(P_S)/\Soc(P_S)$ has at most two direct summands.
Since $S$ occurs with multiplicity three in $P_S$, there has to be a nilpotent
endomorphism of $P_S$ whose image lies in the radical but not in the
socle. Since each other composition factor occurs with multiplicity
one, they must all be in the kernel of such an endomorphism. It
follows that
$\Ext^1_A(S,S)$ is $1$-dimensional, and so $\Rad(P_S)/\Soc(P_S)$ has a direct
summand isomorphic to $S$. Write 
$\Rad(P_S)/\Soc(P_S)=S \oplus X$. If $X=0$, then $A$ has finite
representation type, so we have $X\ne 0$.
Thus the component of the Auslander--Reiten
quiver containing $S$ has the following shape.
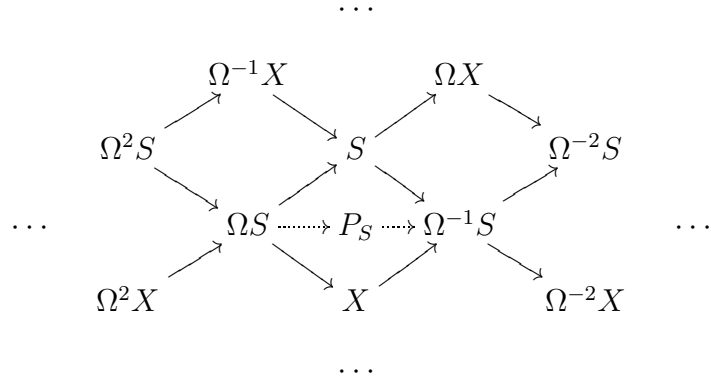
\[ \xymatrix@=4mm{
&&&\cdots\\
&&\Omega^{-1}X\ar[dr]&&\Omega X\ar[dr]\\
&\Omega^2S\ar[ur]\ar[dr]&&S\ar[ur]\ar[dr]&&\Omega^{-2}S\\
\cdots&&\Omega S  \ar[ur]\ar@{.>}[r]\ar[dr] & P_S\ar@{.>}[r] &
\Omega^{-1}S \ar[ur]\ar[dr]&&\cdots\\
&\Omega^2X\ar[ur]&&X\ar[ur]&&\Omega^{-2}X\\
&&&\cdots} \]
The automorphism $\Omega$ sends $S$ to $\Omega S$ and therefore
acts as an automorphism of this stable Auslander--Reiten component. It
is a glide reflection with a horizontal axis, and
its square is the translation. It is easy to check that $A_\infty$ does
not have an automorphism fitting this description, so
this component does not have type $A_\infty$. It now follows from
\autoref{th:Erdmann} that $A$ is not Morita equivalent 
to a block of a finite group algebra in prime characteristic.
\end{proof}

\begin{theorem}\label{th:a11b}
No block of wild representation type of a finite group has Cartan matrix
\[ \begin{pmatrix} a & 1 \\ 1 & b \end{pmatrix} \]
with $2\le a,b\le 7$.
\end{theorem}
\begin{proof}
Let the simple modules be $S$ and $T$. Then $S$ and $T$ have to extend
each other, and so the structures of their projective covers are
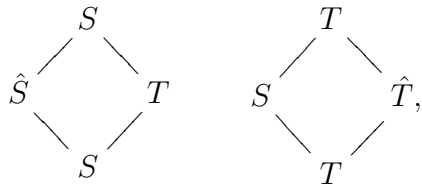
\[ \xymatrix@=4mm{
&S\ar@{-}[dl]\ar@{-}[dr]\\
\hat S\ar@{-}[dr]&&T\ar@{-}[dl]\\
&S} \qquad
\xymatrix@=4mm{
&T\ar@{-}[dl]\ar@{-}[dr]\\
S\ar@{-}[dr]&&\hat T,\ar@{-}[dl]\\
&T} \]
where $\hat S$ and $\hat T$ are modules with $a-2$, respectively $b-2$
composition factors, all isomorphic to $S$, respectively $T$. By
\autoref{th:summands}, $\hat S$ and $\hat T$ are either zero or
indecomposable. The algebras $\End_A(P_S)$ and $\End_A(P_T)$ are
symmetric local algebras of dimension at most seven. So by
\autoref{pr:le7}, they are commutative, and either uniserial
or $6$-dimensional. If both are uniserial, then $A$ is a Brauer tree
algebra, and therefore either of finite representation type or tame
biserial. On the other hand, if either $a$ or $b$ is equal to six, then
the determinant is either prime or $35$. In the former case the block
has cyclic defect, while in the latter case the determinant is not a
prime power, so there is no block of this form.
\end{proof}

\begin{theorem}\label{th:abc}
There is no block of a finite group with Cartan matrix
\[ \begin{pmatrix} a&1&1\\1&b&.\\1&.&c \end{pmatrix} \]
with $a\ge 3$, $b\ge 2$, and $c\ge 2$.
\end{theorem}
\begin{proof}
The structure of the projectives has to be 
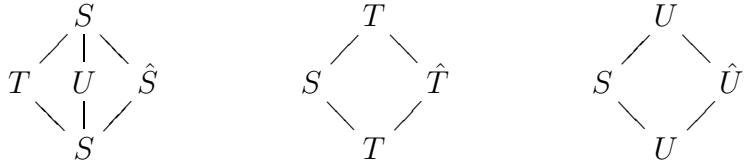
\[ \vcenter{\xymatrix@=3mm{&S\ar@{-}[dl]\ar@{-}[d]\ar@{-}[dr]\\
T\ar@{-}[dr]&U\ar@{-}[d]&\hat S\ar@{-}[dl]\\
&S}}\qquad\qquad
\vcenter{\xymatrix@=3mm{&T\ar@{-}[dl]\ar@{-}[dr]\\
S\ar@{-}[dr]&&\hat T\ar@{-}[dl]\\
&T}}\qquad\qquad
\vcenter{\xymatrix@=3mm{&U\ar@{-}[dl]\ar@{-}[dr]\\
S\ar@{-}[dr]&&\hat U\ar@{-}[dl]\\
&U}}\]
where $\hat S$ has $a-2$ composition  factors, all isomorphic to $S$,
$\hat T$ has $b-2$ composition factors, all isomorphic to $T$,
and $\hat U$ has $c-2$ composition factors, all isomorphic to $U$.
Note that $\hat T$ and $\hat U$ are allowed to be zero. But
since $a\ge 3$, $\hat S$ is not zero, and  $\Rad(P_S)/\Soc(P_S)$ is forced to have at
least three direct summands, contradicting \autoref{th:summands}.
\end{proof}

The following theorem was heavily used in \cite{Sambale:2021a}.

\begin{theorem}\label{contrib}
Let $Q$ be the decomposition matrix of $B$. Then $C=Q^\TT Q$ is the
Cartan matrix of $B$. Let $M:=|D|QC^{-1}Q^\TT$. Then $M$ is an integer
matrix. The number $k_0(B)$ of irreducible height zero characters of
$B$ coincides with the number of diagonal entries of $M$, which are
coprime to $p$. In particular, $D$ is abelian if and only if all
diagonal entries of $M$ are coprime $p$. 
\end{theorem}
\begin{proof}
The equation $C=Q^\TT Q$ is well-known. The second claim follows from
Lemma~4.1 of \cite{Sambale:2020a}. The last claim is a consequence of
the recent solution of Brauer's height zero
conjecture~\cite{Malle/Navarro/SchaefferFry/Tiep}. 
\end{proof}

\section{Small dimensional basic algebras}\label{sec:small}

Dimensions one to seven caused no problems in 
Linckelmann's analysis~\cite{Linckelmann:2018b}.

\subsection*{Dimension eight}
For the Cartan
matrix 
\[ \begin{pmatrix} 3&1\\1&3\end{pmatrix} \]
of determinant $8$,
Linckelmann resorts to knowledge of blocks with defect groups of order $8$.
This case can be eliminated more directly using \autoref{th:310}.

\subsection*{Dimension nine}
In dimension $9$, the 
Cartan matrix that causes difficulty is
\[ \begin{pmatrix} 5&1\\1&2\end{pmatrix}. \]
of determinant $9$. This case was not resolved in Linckelmann's
paper, but Theorem~5.1 of 
Linckelmann and Murphy~\cite{Linckelmann/Murphy:2021a} shows that
this is only possible for a block with cyclic defect.
There, it was proved using some fairly deep results from block theory.
We eliminate it directly as a special case of \autoref{th:a11b}.

\subsection*{Dimension ten}
This did not cause any trouble in Linckelmann's analysis.

\subsection*{Dimension eleven}
The Cartan matrix
\[ \begin{pmatrix} 3&1&1 \\ 1&2&. \\ 1&.&2 \end{pmatrix} \]
was eliminated by Linckelmann using Okuyama's analysis of blocks of
Loewy length three. We can instead apply \autoref{th:310} to
eliminate this case.

\subsection*{Dimension twelve}
Again, this did not cause any trouble in Linckelmann's analysis.

\subsection*{Dimension thirteen}
In dimension $13$, there are two Cartan matrices we wish to
comment on.
The first case we consider is 
\[ \begin{pmatrix} 7&1\\1&4\end{pmatrix}. \]
The determinant is $27$, so we are in characteristic three.
\autoref{th:a11b} shows that this cannot happen for a block of
wild representation type. But if this is the Cartan matrix of a Brauer
tree algebra, then there are two exceptional vertices, so the algebra
is tame biserial. This cannot happen in odd characteristic, so this is ruled out.

The second case we consider is the Cartan matrix
\[ \begin{pmatrix} 5&1&1\\1&2&.\\1&.&2\end{pmatrix}, \]
of determinant $16$. By \autoref{th:abc}, 
no block of a finite group can have
this Cartan matrix.

\subsection*{Dimension fourteen}

In Proposition~3 of \cite{Sambale:2021a}, the first author stated that the Cartan matrix $\bigl(\begin{smallmatrix}
5&2\\2&4
\end{smallmatrix}\bigr)$ belongs to two tame blocks of $\PGL(2,7)$ or $3.M_{10}$. However, Case $(*)$ in Theorem~2.3 of Macgregor~\cite{Macgregor:2022a} states that there might be other Morita equivalences of tame blocks with this Cartan matrix. 
Hence, this case remains open. 

Another Cartan matrix of interest to us is
\[ \begin{pmatrix} 5&1&1\\1&3&.\\1&.&2 \end{pmatrix}. \]
This has determinant $25$, so we are in characteristic five.
Again this violates \autoref{th:abc}, so no block of a finite
group can have this Cartan matrix.

\subsection*{Dimension fifteen}

We begin by formulating an easy lemma that we shall use here, and again in the
case of dimension sixteen.

\begin{lemma}\label{le:l-ge-5}
The dimension $d$ of the basic algebra of a block $B$ is at least
$4l(B)-2$, so $l(B) \le (d+2)/4$.
\end{lemma}
\begin{proof}
Since the Cartan matrix $C$ is
positive definite and indecomposable, its trace is at least $2l(B)$. 
On the other hand, there must be at least $2l(B)-2$ positive entries
off the diagonal, so the sum of the entries of $C$ is at least $4l(B)-2$.
\end{proof}

\begin{proof}[Proof of \autoref{th:dim15}]
Since $15$ is not a prime power, we have $l(B) >1$. 
So using Lemma~\ref{le:l-ge-5}, we have $2\le l(B)\le 4$. 
By \cite[Proposition~2 and its proof]{Sambale:2021a} and \autoref{propfinite}, we
find the two stated blocks of defect $1$. We may now assume that $\det(C)$ is not a prime. 
For $l(B)=2$ there is only one potential Cartan matrix left:
\[C=\begin{pmatrix}
6&3\\3&3
\end{pmatrix}.\]
But $C$ has elementary divisors $3$, $3$ and therefore cannot arise from a block. 
Thus, let $l(B)=3$.  
Since there are at least four positive entries off the diagonal, the
trace of $C$ is bounded by $12$. An individual entry on the diagonal
can therefore be at most $8$. The entries off the diagonal are bounded
by $2$ since otherwise one gets a non-positive minor. This makes it
easy to enumerate all feasible Cartan matrices. Afterwards we remove
those which differ only by permuting the simple modules. This leaves
only the matrices 
\[\begin{pmatrix}
4&.&2\\
.&3&1\\
2&1&2
\end{pmatrix},\quad
 \begin{pmatrix} 6&.&1\\.&3&1\\1&1&2\end{pmatrix} \]
with determinant $8$ and $27$ respectively. This first case is
excluded by \autoref{prop8}. In the second case, \autoref{th:310}
implies that $D$ is cyclic. But then $l(B)\le p-1=2$.  

Finally, if $l(B)=4$, then $C$ is one of the following matrices:
\[
\begin{pmatrix}
2&.&1&.\\
.&2&1&.\\
1&1&2&1\\
.&.&1&3
\end{pmatrix},\quad
\begin{pmatrix}
2&.&1&.\\
.&2&1&1\\
1&1&2&.\\
.&1&.&3
\end{pmatrix}.\]
In the first case, $|D|=8$ and this cannot happen again by
\autoref{prop8}. In the second case $|D|=9$ and $D$ must be cyclic by
\autoref{th:310}. But then $l(B)\le 2$, a contradiction. 
\end{proof}

\subsection*{Dimension sixteen}
This is postponed to \autoref{sec16} below.

\subsection*{Dimension seventeen}

A case of interest in dimension $17$ is the Cartan matrix
\[ \begin{pmatrix} 6&.&1\\.&5&1\\1&1&2\end{pmatrix}. \]
The determinant is $49$, so we are in characteristic
seven. Applying \autoref{pr:le7}, we see that
the heart of each
projective indecomposable has two summands and these are all
uniserial, of
length one, three, or four. The algebra is therefore tame biserial,
and since the characteristic is not two, this therefore cannot be a
block algebra.

There are also some Cartan matrices with four simples, that need to be
considered in dimension $17$, e.\,g.
\[ \begin{pmatrix} 
   2 & . & 1 & . \\
   . &  2 &  1 &  1 \\
   1 &  1 &  2 &  1 \\
   . &  1 &  1 &  3 
\end{pmatrix}.\]
They can all be ruled out easily with \autoref{th:310}.

On the other hand, the principal $2$-block of $\PGL(2,47)$ has defect group $D_{32}$ and Cartan matrix
\[
\begin{pmatrix}
9&2\\2&4
\end{pmatrix}.\]
Using \autoref{contrib} and \cite{Macgregor:2022a}, one can show that this is the only Morita equivalence class for this Cartan matrix. 
The principal $5$-block of $\PSL(2,13)$ has defect group $C_{25}$ and Cartan matrix 
\[\begin{pmatrix}
13&1\\1&2
\end{pmatrix},\]
but here there might be other blocks with non-cyclic defect group and the same Cartan matrix. Similarly, the Cartan matrix
\[\begin{pmatrix}
9&3\\3&2
\end{pmatrix}\]
occurs (at least) for a non-principal block of $2.S_6$ with defect group $C_3^2$. This is the first non-local wild block that we have encountered.

\section{Basic algebras of dimension 16}\label{sec16}

The proof of \autoref{th:dim16} requires the following lemma about a
specific $2$-block of defect four.

\begin{lemma}\label{perfectiso}
Let $B$ be a block with defect group $D\cong D_8\times C_2$ and
$l(B)=2$. Then $B$ is perfectly isometric to the principal block of
$S_4\times C_2$.  
\end{lemma}
\begin{proof}
The blocks with defect group $D_8\times C_2$ were investigated in
Section~9.1 of \cite{Sambale:2014a}. Let 
\[D=\langle x,y\mid x^4=y^2=1,\ yxy^{-1}=x^{-1}\rangle\times\langle
  z\mid z^2=1\rangle\cong D_8\times C_2.\] 
Let $\mathcal{F}$ be the fusion system of $B$ on $D$. Since $l(B)=2$,
we are in case (ab) or (ba) of Lemma~9.3 and Theorem~9.7 in
\cite{Sambale:2014a}. Replacing $y$ by $xy$ if necessary, we may
assume case (ab), i.\,e. $E:=\langle x^2,xy,z\rangle$ is the only
$\mathcal{F}$-essential subgroup and $\Aut_\mathcal{F}(E)\cong
S_3$. Replacing $z$ by $x^2z$ if necessary, we may assume that
$Z(\mathcal{F})=\langle z\rangle$ is the centre of $\mathcal{F}$ and
$\mathfrak{foc}(B)=\langle x^2,xy\rangle$ is the focal subgroup of
$B$.  
Moreover, Theorem~9.7 in \cite{Sambale:2014a} shows that
\[k(B)=k_0(B)+k_1(B)=8+2=10\] 
where $k_i(B)$ denotes the number of irreducible characters of height $i$ in $B$. 

Now observe that the principal block $B_0$ of $S_4\times C_2$ has the
same defect group and the same fusion system as $B$. By Theorem~6.1 of
\cite{Sambale:2020a}, it suffices to show that $B$ and $B_0$ have the
same generalised decomposition matrix up to signs and basic sets.  
By Lemma~9.5 in \cite{Sambale:2014a},
\[\mathcal{R}:=\{1,x,x^2,y,z,xz,x^2z,yz\}\] 
is a set of representatives for the $\mathcal{F}$-conjugacy classes of
$D$. We fix $B$-subsections $(u,b_u)$ for $u\in\mathcal{R}$ such that
$b_u$ has defect group $C_D(u)$.  
We now make use of the Brou\'e--Puig~\cite{Broue/Puig:1980b} $*$-construction. 
By a theorem of Robinson~\cite{Robinson:2008a}, there exist
$\chi_1,\chi_2,\chi_3\in\Irr(B)$ such that 
\begin{align*}
\Irr_0(B)&=\{\lambda*\chi_i:\lambda\in\Irr(D/\mathfrak{foc}(B)),i=1,2\},\\
\Irr_1(B)&=\{\lambda*\chi_3:\lambda\in\Irr(Z(\mathcal{F}))\}
\end{align*}
where $\Irr_i(B)$ is the set of irreducible characters of height $i$
of $B$. By Lemma~10 in \cite{Sambale:2017a}, the generalised
decomposition numbers fulfil
$d_{\lambda*\chi,\phi}^u=\lambda(u)d_{\chi,\phi}^u$ for
$u\in\mathcal{R}$ and $\phi\in\IBr(b_u)$. Hence, it suffices to
determine $d_{\chi_i,\phi}^u$ for $i=1,2,3$. 
Since $D$ is a rational group, these numbers are integers. For $i=1,2$
we have $d_{\chi_i,\phi}^u\ne 0$ by Proposition~1.36 of
\cite{Sambale:2014a}.  
Let $u\in\mathcal{R}\setminus Z(\mathcal{F})$. Then $b_u$ is nilpotent
and $l(b_u)=1$. By the orthogonality relations of generalised
decomposition numbers (see Theorem~1.14 in \cite{Sambale:2014a}), we
have $d^u_{\chi_i,\phi}=\pm1$ for $i=1,2$. We may choose basic sets
such that $d^u_{\chi_1,\phi}=1$ for all $u\in\mathcal{R}\setminus
Z(\mathcal{F})$.  
If $u\in\{x,y,xz,yz\}$, then $b_u$ has defect $3$ and $d^u_{\chi_3,\phi}=0$. 
We may choose $\chi_2$ such that
$d_{\chi_2,\phi}^x=1=d_{\chi_2,\phi}^{xz}$. The orthogonality between
$x$, $y$ and $xz$, $yz$ shows that
$d_{\chi_2,\phi}^y=-1=d_{\chi_2,\phi}^{yz}$.  

It remains to consider $u\in\{x^2,z,x^2z\}$. Replacing $\chi_3$ by
$-\chi_3$ if necessary, we may assume that $d_{\chi_3,\phi}^{x^2}=2$.  
Recall that $b_z$ dominates a unique block $\overline{b_z}$ of
$C_G(z)/\langle z\rangle$ with defect group $D/\langle z\rangle\cong
D_8$. The Cartan matrix of $\overline{b_z}$ is
$\overline{C_z}:=\bigl(\begin{smallmatrix} 
3&1\\1&3
\end{smallmatrix}\bigr)$ up to basic sets by \autoref{prop8}. Hence,
the Cartan matrix of $b_z$ is $2\overline{C_z}$ up to basic sets. We
may choose a basic set and $\alpha=\pm1$ such that
$d_{\chi_1,.}^z=(1,0)$, $d_{\chi_2,0}^z=(0,\alpha)$ and
$d_{\chi_3,.}^z=(1,1)$ (interchanging $\chi_3$ and $\lambda*\chi_3$ if
necessary). The orthogonality between $z$ and $x^2z$ implies
$d_{\chi_3,\phi}^{x^2z}=-2$ and $d_{\chi_2,\phi}^{x^2z}=\alpha$.  
In order to determine $\alpha$ and $\beta:=d_{\chi_2,\phi}^{x^2}$ we
use the contribution matrices  
\[(m_{\chi,\psi}^u)_{\chi,\psi\in\Irr(B)}:=16Q_uC_u^{-1}Q_u^\TT\in\bZ^{10\times 10},\]
where $Q_u=(d^u_{\chi,\phi})$ and $C_u=Q_u^\TT Q_u$ is the Cartan
matrix of $b_u$.   
We compute 
\[m_{\chi_1,\chi_2}^u=\begin{cases}
2&\text{if }u\in\{x,xz\},\\
-2&\text{if }u\in\{y,yz\},\\
\alpha&\text{if }u=x^2z,\\
\beta&\text{if }u=x^2.
\end{cases}\]
By restricting a generalised character of $S_4\times C_2$, we obtain
an $\mathcal{F}$-invariant generalised character $\lambda$ of $D$ such
that $\lambda(1)=\lambda(z)=0$,
$\lambda(x)=\lambda(xz)=\lambda(x^2)=\lambda(x^2z)=4$ and
$\lambda(y)=\lambda(yz)=2$.
Then 
\[0=\sum_{u\in\mathcal{R}}\lambda(u)m_{\chi_1,\chi_2}^u
 =4(2+2+\alpha+\beta)+2(-2-2)\]
by \cite[p. 684]{Ardito/Sambale:2022a}. It follows that $\alpha=\beta=-1$.
This completely determines the generalised decomposition matrices for
non-trivial subsections as follows:  
\[\begin{array}{c|*{8}{r}}
&x&y&xz&yz&x^2&x^2z&\multicolumn{2}{c}{z}\\\hline
\chi_1&1&1&1&1&1&1&1&.\\
&-1&-1&-1&-1&1&1&1&.\\
&1&1&-1&-1&1&-1&-1&.\\
&-1&-1&1&1&1&-1&-1&.\\
\chi_2&1&-1&1&-1&-1&-1&.&-1\\
&-1&1&-1&1&-1&-1&.&-1\\
&1&-1&-1&1&-1&1&.&1\\
&-1&1&1&-1&-1&1&.&1\\\hline
\chi_3&.&.&.&.&2&-2&1&1\\
&.&.&.&.&2&2&-1&-1
\end{array}
\]
Now the claim follows from Theorem~6.1 of \cite{Sambale:2020a}.
\end{proof}

\begin{proof}[Proof of \autoref{th:dim16}]
As before, let $C$ be the Cartan matrix of $B$.\smallskip

\textbf{Case~1:} $l(B)=1$.\\
Here, $|D|=16$. If $B$ is nilpotent, then $A\cong kD$ by Puig's
theorem (see Theorem~1.30 in \cite{Sambale:2014a}). By partial solutions on the modular isomorphism problem,
these algebras are pairwise non-isomorphic (see Lemma~14.2.7 in
\cite{Passman:1977a}). 
If $B$ is non-nilpotent, then $D$ must be elementary abelian and the
inertial index of $B$ is $9$ (see Theorem~13.2 and the proof of
Theorem~13.6 in \cite{Sambale:2014a}). By Eaton's
classification~\cite{Eaton:16}, $A$ is Morita equivalent to a
non-principal block of $H$ as given in the statement. In total we
obtain 15 isomorphism types of basic algebras with $l(B)=1$.\smallskip

\textbf{Case~2:} $l(B)=2.$\\
By \autoref{proptame}, $B$ cannot be a tame block. 
Using Proposition~2 of \cite{Sambale:2021a}, it is easy to see that
$C$ is one of the following matrices 
\[
\begin{array}{cccccc}
C&\begin{pmatrix}
12&1\\1&2
\end{pmatrix}&\begin{pmatrix}
11&1\\1&3
\end{pmatrix}&\begin{pmatrix}
10&2\\2&2
\end{pmatrix}&\begin{pmatrix}
6&2\\2&6
\end{pmatrix}&\begin{pmatrix}
5&3\\3&5
\end{pmatrix}\\
|D|&23&32&8&16&16
\end{array}
\]
The first case occurs for the principal $23$-block of
$\PSL(2,137)$. Here $A$ is uniquely determined as a Brauer tree
algebra. 
In the second case, $B$ must have finite representation type by
\autoref{th:310}. Then $D\cong C_{32}$ and $B$ would be nilpotent
since $\Aut(D)$ is a $2$-group.  
In the third case, $D\cong D_8$ since otherwise $l(B)\ne 2$. But then
$B$ would be tame. Now consider the fifth case. 
The possible decomposition matrices of $B$ are
\[
\begin{pmatrix}
2&1\\
.&1\\
.&1\\
.&1\\
1&1
\end{pmatrix},\quad
\begin{pmatrix}
1&2\\
1&.\\
1&.\\
1&.\\
1&1
\end{pmatrix},\quad
\begin{pmatrix}
1&.\\
1&.\\
.&1\\
.&1\\
1&1\\
1&1\\
1&1
\end{pmatrix}.
\]
By \autoref{contrib} we obtain $k_0(B)=4$ in all cases. Since $B$
satisfies the Alperin--McKay conjecture (see Theorem~13.6 of
\cite{Sambale:2014a}), we have $k_0(B_D)=4$ where $B_D$ is the Brauer
correspondent of $B$ in $N_G(D)$. Recall that $B_D$ dominates a block
$\overline{B_D}$ of $N_G(D)/D'$ with abelian defect group
$D/D'$. Hence, 
\[k(\overline{B_D})=k_0(\overline{B_D})\le k_0(B_D)=4\]
and $|D/D'|=4$ by Theorem~1.31 in \cite{Sambale:2014a}. But then $D$
has maximal nilpotency class and $B$ would be tame.  

It remains to deal with the Cartan matrix $C=\bigl(\begin{smallmatrix}
6&2\\2&6
\end{smallmatrix}\bigr)$, where $|D|=16$. Since $B$ is not tame, we have $k_0(B)>4$ as
seen above. The possible decomposition matrices of $B$ are: 
\[
\begin{pmatrix}
2&1\\
1&.\\
1&.\\
.&1\\
.&1\\
.&1\\
.&1\\
.&1
\end{pmatrix},\quad
\begin{pmatrix}
1&.\\
1&.\\
1&.\\
1&.\\
.&1\\
.&1\\
.&1\\
.&1\\
1&1\\
1&1
\end{pmatrix}.
\]
In the first case $k_0(B)=k(B)$ and $D$ is abelian by
\autoref{contrib}. However, there is no such block with $l(B)=2$. 
Therefore, the second matrix must be the decomposition matrix of
$B$. In particular, $k(B)=k_0(B)+k_1(B)=8+2=10$. Using the results in
\cite[Chapters~8, 9]{Sambale:2014a}, one can exclude metacyclic defect
groups and $Q_8\times C_2$ or $Q_8*C_4$ for $D$. The remaining cases
are $D\cong D_8\times C_2$ or the minimal non-abelian group
$\mathtt{SmallGroup}(16,3)$. By \autoref{perfectiso} and Theorem~9 in
\cite{Sambale:2016a}, $B$ is perfectly isometric to principal block of
$H_1:=S_4\times C_2$ or $H_2:=\mathtt{SmallGroup}(48,30)\cong
A_4\rtimes C_4$. In particular, $Z(A)\cong Z(B)\cong Z(kH_i)$ for
$i=1$ or $2$ (see Theorem~4.4 in \cite{Sambale:2020a}). One can show
with MAGMA~\cite{Bosma/Cannon/Playoust:1997a} that 
\begin{equation}\label{eq:ZA}
Z(kH_1)\cong Z(kH_2)\cong k[w,x,y,z]/(w^2x,w^2y,
w^2+z^2,x^2,xy,xz,y^2,yz,z^3)
\end{equation}
with basis $1$, $w$, $x$, $y$, $z$, $w^2=z^2$, $wx$, $wy$,
$wz$, $w^3=wz^2$. 
The centre $Z(A)$ is the subset of
$\End(P_S)\times\End(P_T)$ consisting of the elements that annihilate the
homomorphisms from $P_S$ to $P_T$ and from $P_T$ to $P_S$. Now
$\End(P_S)\times\End(P_T)$ has dimension $12$. The idempotents $\ep_S$
and $\ep_T$ are not in $Z(A)$ but their sum is. Since $Z(A)$ has
dimension $10$, it follows that the radical
$J(Z(A))$ has codimension one in $J(\End(P_S))\times J(\End(P_T))$.
By \eqref{eq:ZA}, $J(Z(A)$ is indecomposable, 
the projections $Z(A) \to \End(P_S)$ and $Z(A)\to\End(P_T)$
are surjective and $Z(A) \to \End(P_S) \times \End(P_T)$ is
injective. This implies that $\End(P_S)$ and $\End(P_T)$ are
$6$-dimensional commutative Gorenstein rings. Using \eqref{eq:ZA}, 
we claim that the only $6$-dimensional Gorenstein quotient of $Z(A)$
is $Z(A)/(x,y)\cong k[w,z]/(w^2+z^2,z^3)$. For if $x$ or $y$ has
non-zero image in a Gorenstein quotient, then the socle has to be
divisible by $x$ or $y$. This forces the Loewy length to be three. Since
the socle of a Gorenstein ring is $1$-dimensional, this means that the
second Loewy layer has to be $4$-dimensional. It is easy to see that
there is no such Gorenstein quotient.\smallskip

\textbf{Case~3:} $l(B)=3$.\\
By Proposition~2 in \cite{Sambale:2021a}, there are no such blocks of
defect $1$. By \autoref{proptame} and Theorem~2.1 in Macgregor~\cite{Macgregor:2022a}, there is just one Morita
equivalence class of tame blocks. Here $D\cong D_8$ and $A$ is Morita
equivalent to the principal block of $\PSL(2,7)\cong\GL(3,2)$. Now
suppose that $B$ is not tame. As explained in the proof of
\autoref{th:dim15}, it is easy to make a list of potential Cartan
matrices: 
\[
\begin{array}{cccccc}
C&\begin{pmatrix}
5&1&.\\
1&3&2\\
.&2&2
\end{pmatrix}&
\begin{pmatrix}
3&2&1\\
2&3&1\\
1&1&2
\end{pmatrix}&
\begin{pmatrix}
5&1&2\\
1&3&.\\
2&.&2
\end{pmatrix}&
\begin{pmatrix}
6&1&1\\
1&2&1\\
1&1&2
\end{pmatrix}&
\begin{pmatrix}
7&.&1\\
.&3&1\\
1&1&2
\end{pmatrix}
\\
|D|&4&8&8&16&32
\end{array}
\]
The first three candidates are excluded by \autoref{prop8}. In the
fifth case, $B$ has finite representation type by
\autoref{th:310}. But then $B$ would be nilpotent since $p=2$.  
Now consider case four, where $|D|=16$. The possible decomposition
matrices are: 
\[
\begin{pmatrix}
2&.&.\\
1&1&1\\
.&1&.\\
.&.&1\\
1&.&.
\end{pmatrix},\quad
\begin{pmatrix}
1&1&.\\
1&.&1\\
.&1&1\\
1&.&.\\
1&.&.\\
1&.&.\\
1&.&.
\end{pmatrix},\quad
\begin{pmatrix}
1&1&1\\
.&1&.\\
.&.&1\\
1&.&.\\
1&.&.\\
1&.&.\\
1&.&.\\
1&.&.
\end{pmatrix}
\]
The first two cases yields $k_0(B)=4$. But then $B$ must be a tame
block by the arguments above for $l(B)=2$. 
In the final case, $k_0(B)=k(B)=8$ and $D$ is abelian. Here $D$ cannot
be of type $C_4\times C_2^2$, because this would yield an elementary
divisor $2$ of $C$. Consequently, $D$ is elementary abelian. However,
this is excluded by Eaton~\cite{Eaton:16}.
\smallskip

\textbf{Case~4:} $l(B)=4$.\\
By \autoref{propfinite}, there three (potential) blocks of defect $1$:
two $5$-blocks with multiplicity $m=1$ and a $13$-block with
multiplicity $3$. The $5$-blocks occur in $S_5$ and $\Sz(8)$ by
\cite{Hiss/Lux:1989a}. The $13$-block is Morita equivalent to the principal
block of $\GL(4,5)$ by \cite{Fong/Srinivasan:1980a}. In the general case we
enumerate the possibilities for $C$. 
There are at least six positive off-diagonal entries of $C$ and
therefore the trace is bounded by $10$. The diagonal entries are
bounded by $4$ while the off-diagonal entries can be at most $3$. 
This only leaves the case
\[C=\begin{pmatrix}
2&1&.&1\\
1&2&1&.\\
.&1&3&.\\
1&.&.&3
\end{pmatrix}\]
where $|D|=16$. Here $B$ is not tame since $l(B)=4$. Hence, this block
is excluded again by \autoref{th:310} since $p=2$.\smallskip

\textbf{Case~5:} $l(B)\ge 5$.\\
This cannot happen, by Lemma~\ref{le:l-ge-5}.
\end{proof}

\subsection*{Acknowledgements}

It is a pleasure for the first author to acknowledge the support of
the Hausdorff Institute of Mathematics
in Bonn for their hospitality during the programme ``Spectral Methods
in Algebra, Geometry, and Topology,'' where part of this research was carried
out. It is also a pleasure to thank Charles Eaton and Markus
Linckelmann for sharing their expertise and encouragement.

\end{document}